\documentclass[10pt,a4paper,reqno]{amsart}
\usepackage{epsfig}
\usepackage{color}
\usepackage{amssymb,amsmath,amsthm,amstext,amsfonts}
\usepackage{amssymb,amsmath,amsthm,amstext,amsfonts}
\usepackage{amsmath,amstext,amsthm,amsfonts}
\usepackage{color}
\usepackage[mathscr]{eucal}
\usepackage{dsfont}

\usepackage{psfrag}
\usepackage{url}
\usepackage{epstopdf}
\usepackage{epic,eepic}
\usepackage{upgreek}
\usepackage{amssymb}
\usepackage{mathrsfs}
\usepackage{verbatim}
\usepackage{mathrsfs}
\usepackage{amstext}
\usepackage{amsthm}
\usepackage{amssymb}
\usepackage{graphicx}

\usepackage[colorlinks=true, linkcolor=blue, urlcolor=red, citecolor=blue]{hyperref}

\makeatletter \@addtoreset{equation}{section} \makeatother

\renewcommand\thetable{\thesection.\@arabic\c@table}

\theoremstyle{plain}
\newtheorem{maintheorem}{Theorem}
\newtheorem{theorem}{Theorem }[section]
\newtheorem{proposition}{Proposition}[section]

\newtheorem{lemma}{Lemma}[section]
\newtheorem{maincorollary}{Corollary}

\theoremstyle{definition} \theoremstyle{remark}
\newtheorem{remark}[theorem]{Remark}

\newtheorem{definition}[theorem]{Definition}

\newcommand{\vep}{\varepsilon}

\newcommand{\N}{\mathbb{N}}

\newcommand{\supp}{\operatorname{supp}}

\def\N{\mathbb{N}}

\begin{document}
\title{Some variational principles for the  metric mean dimension of a  semigroup action}

\author[F. Rodrigues]{Fagner B. Rodrigues*}
\thanks{*Corresponding author e-mail: fagnerbernardini@gmail.com}
\address{Departamento de Matem\'atica, Universidade Federal do Rio Grande do Sul, Brazil.}
\email{fagnerbernardini@gmail.com}

\author[T. Jacobus]{Thomas Jacobus}
\address{Departamento de Matem\'atica, Universidade Federal do Rio Grande do Sul, Brazil.}
\email{jacobus.math@gmail.com }

\author[M. Silva]{Marcus V. Silva}
\address{Departamento de Matem\'atica, Universidade Federal do Rio Grande do Sul, Brazil.}
\email{marcus423@gmail.com }

\date{\today}

\begin{abstract}
In this manuscript  we show that the metric mean dimension of a free semigroup action satisfies 
three variational principles: (a) the first one is based on a definition of Shapira's entropy, introduced in \cite{SH} for a singles dynamics
and extended for a semigroup action in this note; (b) the second one treats about a definition of Katok's entropy for a free semigroup action introduced in \cite{CRV-IV}; (c) lastly we consider the local entropy function for a free semigroup action and show that the metric mean dimension satisfies a variational principle in terms of such function. Our results are inspired in the ones obtained by \cite{LT2019}, \cite{VV}, \cite{GS1}  and \cite{RX}.

\end{abstract}

\maketitle

\section{Introduction}
The  aim of this note is to explore the notion of metric mean dimension for a free semigroup action. The notion of metric mean dimension for a dynamical system $f:(X,d)\to (X,d)$, denoted by $\text{mdim}_M(X,\phi,d')$,  was   introduced in \cite{LW2000} and may be related to the problem of whether or not a given dynamical system can be embedded in the shift space $(([0,1]^\mathbb{N})^{\mathbb{Z}}, \sigma)$.
It refines the topological entropy for systems with infinite entropy,
which, in the case of a manifold of dimension greater than one,  form a residual subset of the set consisting of   homeomorphisms   defined on the manifold (see \cite{Yano}). In fact,
 every system with finite topological entropy has metric mean dimension equals to zero.
The metric mean dimension depends on the metric   $d$, therefore it is not a  topological invariant. However, for a metrizable topological space $X$,  $\text{mdim}_M(X,\phi)=\inf _{d'}\text{mdim}_M(X,\phi,d')$ is invariant under topological conjugacy, where the infimum  is taken over all the metrics on $X$ which induce the topology on $X$. By the other hand, as  showed in  \cite{LT2019} and in \cite{VV}, the metric mean dimension is 
strongly related with the ergodic behaviour of the system, since it satisfies a kind of variational principle.

In \cite{CRV-IV} the authors considered the  compact metric space $(Y^\mathbb N,D)$ and  $(X,d)$, where $(Y,d_Y)$ is a compact metric space and $D$ is the product metric induced by $d_Y$. In this setting  they introduced the notion of metric mean dimension for a free semigroup action and proved that for certain classes of random walks; the ones induced by homogeneous probability measures on $Y$, it is possible to obtain a kind of 
Bufetov's formula (see \cite{Buf} for Bufetov's formula for the topological entropy of a free semigroup action). 

Our main goal here is to consider a compactly  generated free semigroup of continuous maps acting on a compact metric and prove that the metric mean dimension satisfies several  variational principles: (a) the first one is based on a definition of Shapira's entropy, introduced in \cite{SH} for a singles dynamics
and extended for a semigroup action in this note; (b) the second one treats about a definition of Katok's entropy for a free semigroup action introduced in \cite{CRV-IV}; (c) lastly we consider the local entropy function for a free semigroup action and show that the metric mean dimension satisfies a variational principle in terms of such function. Our results are inspired in the ones obtained by \cite{LT2019}, \cite{VV}, \cite{GS1}  and \cite{RX}.  As a second objective, we extend the definition of metric mean dimension when we have a compactly generated semigroup and the the topological entropy is the one defined in \cite{GLW}. In this context we  obtain a partial variational principle for the metric mean dimension.

This paper is organized as follows. In Section \ref{section1} we present the main definitions and the main results. In Section \ref{section3} we 
recall some results and definitions about box dimension, homogeneous measures and $G$-homogeneous measures. In Section \ref{section4} we prove the main theorems.

\section{Definitions and Main results}\label{section1}

We start recalling the main concepts we use and describing the systems we will work with.

\subsection{Metric mean dimension of a map}\label{sse.mmd}

Let $(X,d)$ be a compact metric space. Given a continuous map $f \colon X \to X$ and a non-negative integer $n$, define the dynamical metric $d_n \colon X \times X \, \to \,[0,\infty)$ by
$$d_n(x,z)=\max\,\Big\{d(x,z),\,d(f(x),f(z)),\,\dots,\,d(f^{n}(x),f^{n}(z))\Big\}$$
which generates the same topology as $d$. Having fixed $\varepsilon>0$, we say that a set $E \subset X$ is $(n,\varepsilon)$--separated by $f$ if $d_n(x,z) > \varepsilon$ for every $x,z \in E$. In the particular case of $n=1$, we will call such a set $\varepsilon$--separated. Denote by $s(f,n,\varepsilon)$ the maximal cardinality of all $(n,\varepsilon)$--separated subsets of $X$ by $f$. Due to the compactness of $X$, the number $s(f,n,\varepsilon)$ is finite for every $n \in \mathbb{N}$ and $\varepsilon >0$. We say that $R\subset X$ is a $(n,\varepsilon)$--spanning set if for any $x \in X$ there exists $z \in R$ such that $d_n(x,z)<\varepsilon$. When $n=1$, we say that the set is $\varepsilon$--spanning. Let $r(n,\varepsilon)$ be the minimum cardinality of the $(n,\varepsilon)$--spanning subsets of $X$.

\begin{definition}\label{de.metric-mean}
The \emph{lower metric mean dimension} of $f$ with respect to the fixed metric $d$ is given by
$$\underline{\text{mdim}}_M\,\Big(X,f, d\Big) = \liminf_{\varepsilon \,\to \,0^+} \,\frac{h(f,\varepsilon)}{|\log \varepsilon|}$$
where
$$h(f,\varepsilon) = \limsup_{n\, \to\, \infty}\,\frac{1}{n}\,\log s(f,n,\varepsilon).$$
Similarly, the \emph{upper metric mean dimension} of $f$ with respect to $d$ 
is the limit
$$\overline{\text{mdim}}_M\,\Big(X,f,d\Big) = \limsup_{\varepsilon\,\to\, 0^+} \,\frac{h(f,\varepsilon)}{|\log \varepsilon|}.$$
\end{definition}
\noindent Clearly, $\underline{\text{mdim}}_M\,\Big(X,f,d\Big) = \overline{\text{mdim}}_M\,\Big(X,f,d\Big) = 0$ whenever
the topological entropy of $f$, given by $h_{\text{top}}(f) = \lim_{\varepsilon \, \to \, 0^+}\,h(f,\varepsilon)$, is finite.

\subsection{Compactly generated semigroup action of continuous maps}\label{se.semigroup-action}

Let $(X,d)$ and $(Y,d_Y)$ be compact metric spaces and $(g_y)_{y \,\in\, Y}$ be a family of continuous maps $g_y \colon X \to X$. Denote by $G$ the free semigroup having the set $G_1 = \{g_y \colon \,\,y \,\in\, Y\}$ as generator, where the semigroup operation $\circ$ is the composition of maps. Let $\mathbb{S}$ be the induced free semigroup action
$$\begin{array}{rccc}\label{de.semigroup-action}
\mathbb{S} \colon & G \times X & \to & X \\
	& (g,x) & \mapsto & g(x)
\end{array}$$
which is said to be compactly generated by $Y$, and denote by $T_G$ the associated skew product given by
\begin{equation}\label{de.Skew-product}
\begin{array}{rccc}
T_G \colon & Y^{\mathbb N}  \times X & \to & Y^{\mathbb N} \times X \\
	& (\omega,x) & \mapsto & \Big(\sigma(\omega), g_{\omega_1}(x)\Big),
\end{array}
\end{equation}
where $\omega=(\omega_1,\omega_2, \dots)$ is an element of the full unilateral space of sequences $Y^{\mathbb N}$ and $\sigma$ denotes the shift map acting on $Y^{\mathbb N}$. It will be a standing assumption that $T_G$ is a continuous map. If for every $n \in \mathbb{N}$ and $\omega=(\omega_1,\omega_2, \dots) \in Y^{\mathbb N}$ we write
$$f_\omega^n = g_{\omega_n}\, \dots \,g_{\omega_1}$$
then
$$T_G^n(\omega,x) = \Big(\sigma^n(\omega), f_\omega^n(x)\Big).$$

\medskip

Consider the set $G_1^* = G_1 \setminus \{id\}$ and, for each $n \in \mathbb{N}$, let $G_n^*$ denote the space of concatenations of $n$ elements in $G_1^*$. Similarly, define $G=\bigcup_{n\,\in\, \mathbb N_0} G_n$, where $G_0=\{id\}$ and $\underline g \in G_n$ if and only if $\underline g=g_{\omega_n}\, \dots\, g_{\omega_2}\, g_{\omega_1}$, with $g_{\omega_j} \in G_1$ (for notational simplicity's sake we will use $g_j \, g_i$ instead of the composition $g_j\,\circ\, g_i$).
In what follows, we will assume that the generator set $G_1$ is minimal, meaning that no function $g_y \in G_1$, for $y \in Y$, can be expressed as a composition of the remaining generators. To summon an element $\underline{g}$ of $G^*_n$, we will write $|\underline{g}|=n$ instead of $\underline g\in G^*_n$. Each element $\underline g$ of $G_n$ may be seen as a word which originates from the concatenation of $n$ elements in $G_1$. Yet, different concatenations may generate the same element in $G$. Nevertheless, in the computations to be done, we shall consider different concatenations instead of the elements in $G$ they create.

\subsection{Random walks}

A random walk $\mathbb{P}$ on $Y^{\mathbb N}$ is a Borel probability measure in this space of sequences which is invariant by the shift map $\sigma$. For instance, we may consider a finite subset $F=\{p_1, \dots, p_k\}$ of $Y$, a probability vector $(a_1, \cdots,a_k)$ (that is, a selection of positive real numbers $a_i$ such that $\sum_{i=1}^k \, a_i = 1$), the probability measure $\nu=\sum_{i=1}^{k}\,a_i\,\delta_{p_i}$ on $F$ and the Borel product measure $\mathbb{P}_\nu = \nu^{\mathbb N}$ on $Y^{\mathbb N}$. Such a $\mathbb{P}_\nu$ will be called a \emph{Bernoulli measure}, which is said to be \emph{symmetric} if $a_i=\frac{1}{k}$ for every $i \in \{1, \cdots, k\}$, in which case we denote it by $\mathbb{P}_k$. If $Y$ is a Lie group, a natural symmetric random walk is given by $\nu^{\mathbb N}$ where $\nu$ is the Haar measure. We denote by $\mathscr{P}(Y^{\mathbb N})$ the space of Borel probability measures on $Y^{\mathbb N}$ and by $\mathscr{P}_B(Y^{\mathbb N})$ its subset of Bernoulli elements. It will be clear later on that the role of each random walk is to point out a particular complex feature of the dynamics, here defined in terms of either the topological entropy (definition in Subsection~\ref{se.topo-entropy}) or the metric mean dimension (definition in Subsection~\ref{mdm-semigroup}). 

\subsection{Topological entropy of an action $\mathbb{S}$}\label{se.topo-entropy}

Given $\varepsilon>0$ and $\underline g := g_{\omega_{n}} \dots g_{\omega_2} \, g_{\omega_1}\in G_n$, the $n$th-dynamical ball $B_n(x,\underline g,\varepsilon)$ is the set
\begin{align}\label{eq:dynball}
B_n(x,\underline g,\varepsilon)\nonumber
	&:= \Big\{z \in X: d(\underline g_{\,j} (z), \underline g_{\,j} (x) ) \leqslant \varepsilon, \; \forall\, 0\leqslant j \leqslant n \Big\}
\end{align}
where, for every $0 \leqslant j \leqslant n$, the notation $\underline g_{\,j}$ stands for the concatenation $g_{\omega_{j}} \dots g_{\omega_2} \, g_{\omega_1}$ in $G_j$, and $\underline g_0=id$. Observe that this is a classical ball with respect to the dynamical metric $d_{\underline g}$ defined by
\begin{equation}\label{eq:dg}
 d_{\underline g} (x, z):=   \max_{0\,\leqslant\, j \,\leqslant \,n } \, d(\underline{g}_{\,j}(x),\underline{g}_{\,j}(z)).
\end{equation}
Notice also that both the dynamical ball and the dynamical metric depend on the underlying concatenation of generators $g_{\omega_n} \dots g_{\omega_1}$ and not on the semigroup element $\underline g$, since the latter may have distinct representations.

Given $\underline g = g_{\omega_n} \dots g_{\omega_1} \in G_n$, we say that a set $K \subset X$ is \emph{$(\underline g, n, \varepsilon)$--separated} if $d_{\underline g}(x,z) > \varepsilon$ for any two distinct elements $x,z \in K$. The largest cardinality of any $(\underline g, n,\varepsilon)$--separated subset on $X$ is denoted by $s(\underline g, n, \varepsilon)$ (or, equivalently, $s(g_{\omega_n} \dots g_{\omega_1}, n,\varepsilon)$). A set $K \subset X$ is said to be \emph{$(\underline g, n, \varepsilon)$--spanning} if for every $x \in X$ there is $k \in K$ such that $d_{\underline g}(x,k) \leqslant \varepsilon$. The smallest cardinality of any $(\underline g, n,\varepsilon)$--spanning subset on $X$ is denoted by $b(\underline g, n, \varepsilon)$ (or $b(g_{\omega_n} \dots g_{\omega_1}, n,\varepsilon)$).

\medskip

\begin{definition}\label{de.top-entropy}
The \emph{topological entropy of the semigroup action $\mathbb{S}$} with respect to a fixed set of generators $G_1$ and a random walk $\mathbb{P}$ in $Y^{\mathbb N}$ is given by
$$h_{\text{top}}(\mathbb{S}, \mathbb{P}) :=\lim_{\varepsilon \,\to \,0^+}\,\limsup_{n\,\to\,\infty}\,\frac1n\,\log
	\int_{Y^{\mathbb N}} \, s(g_{\omega_n} \dots g_{\omega_1}, n,\varepsilon) \, d\,\mathbb{P}(\omega)$$
where $\omega = \omega_1 \,\omega_2 \cdots \omega_n \cdots$. The \emph{topological entropy of the semigroup action $\mathbb{S}$} is then defined by
$$h_{\text{top}}(\mathbb{S}) = \sup_{\mathbb{P}}\,h_{\text{top}}(\mathbb{S}, \mathbb{P}).$$
\end{definition}

We observe that the semigroup may have multiple generating sets, and the dynamical or ergodic properties (as the topological entropy) depend on the chosen generator set. More information regarding these concepts in the case of finitely generated free semigroup actions may be read in \cite{CRVI, CRVII,CRVIII}.


\subsection{Entropy function}
Let $(X, d)$ be a compact metric space. For each $\varepsilon>0$ and $x\in X$, define 
\[
h_d(x,\varepsilon)=\inf\{B(K,\mathbb S,\varepsilon): K \text{ is compact neighbourhood of }x\},
\]
where  
\[
B(K,\mathbb S,\varepsilon)=\limsup_{n\to\infty}\frac{1}{n}\log\left(\int_{\Sigma_p^+}b(K, g_{\omega_n}\dots g_{\omega_1},\varepsilon)\;d\mathbb P(\omega)\right),
\]
and $b(K, g_{\omega_n}\dots g_{\omega_1},\varepsilon)$ denotes the minimum cardinality of a $(g_{\omega_n}\dots g_{\omega_1},\varepsilon)$-spanning set. 
As $h_d(x,\varepsilon)$ increases as $\varepsilon$ decreases to zero, it is well defined the following
\begin{align}\label{eq:entropy-function}
h_d(x)=\lim_{\varepsilon\to0^+}h_d(x,\varepsilon)
\end{align}
and it is less or equal to $h_{top}(X,\mathbb S)$. In fact, it depends only on the topology of $X$ and
we can denote by $h_{top}(x)$.

\begin{definition}\label{def:entropy-function}
Let $\mathbb S:G\times X\to X$ be a continuous finitely generated free semigroup
action. The function $h_{top}:X\to [0,h_{top}(X,\mathbb S)]$, $x\mapsto h_{top}(x)$
is called the entropy function of $\mathbb S$.
\end{definition}

Since $B(K,\mathbb S,\varepsilon)\leq S(K,\mathbb S,\varepsilon)\leq B(K,\mathbb S,\varepsilon\slash 2)$, we have
\[
h_{top}(x)=\lim_{\varepsilon\to0}\inf\{S_d(K,\mathbb S,\varepsilon): K \text{ is compact neighbourhood of }x\}.
\]
By \cite[Theorem C]{RJS} we have that 
$$
\sup_{x\in X} \lim_{\varepsilon\to 0}h_d(x,\varepsilon)=h_{top}(\mathbb S,\mathbb P).
$$
\subsection{Metric mean dimension of a semigroup action}\label{mdm-semigroup}
  Let $(X,d)$ be a compact metric space and $\mathbb{S}$ be the free semigroup action induced on $(X,d)$ by a family of continuous maps $(g_y \colon X \to X)_{y \,\in\, Y}$. The following definition for the semigroup setting was introduced in \cite{CRV-V}.
\begin{definition}\label{de.mmd}
The \emph{upper and lower metric mean dimension of the free semigroup action $\mathbb{S}$} on $(X,d)$ with respect to a fixed set of generators $G_1$ and a random walk $\mathbb{P}$ in $Y^{\mathbb N}$ are given respectively by
\begin{eqnarray*}
\displaystyle\overline{\text{mdim}}_M \,\Big(X,\mathbb S, d, \mathbb{P}\Big) &=& \limsup_{\varepsilon\,\to\,0^+}\,\frac{h(X,\mathbb S,\mathbb{P},\varepsilon)}{-\log\varepsilon}\\
\displaystyle\underline{\text{mdim}}_M \,\Big(X,\mathbb S, d, \mathbb{P}\Big) &=& \liminf_{\varepsilon\,\to\,0^+}\,\frac{h(X,\mathbb S,\mathbb{P},\varepsilon)}{-\log\varepsilon}
\end{eqnarray*}
where
\begin{equation}\label{eq:defhep}
h(X,\mathbb S,\mathbb{P},\varepsilon)=\displaystyle\limsup_{n\,\to\,\infty}\,\displaystyle\frac{1}{n}\,\log\displaystyle\int_{Y^{\mathbb N}}\,s(g_{\omega_n}\dots g_{\omega_1},n,\varepsilon)\,\,d\mathbb{P}(\omega).
\end{equation}
\end{definition}
 Our first result shows that the metric mean dimension of a semigroup action may be computed in terms of the entropy function.
 \begin{maintheorem}\label{thm6}
  Let $(X,d)$ be a compact metric space and $\mathbb{S}$ be the free semigroup action induced on $(X,d)$ by a family of continuous maps $(g_y \colon X \to X)_{y \,\in\, Y}$. Then
  $$
\displaystyle\overline{\text{mdim}}_M \,\Big(X,\mathbb S, d, \mathbb{P}\Big)=\limsup_{\varepsilon\,\to\,0^+}\,\frac{\displaystyle\sup_{x\in X}h_d(x,\varepsilon)}{-\log\varepsilon}, \\
$$
for every $\mathbb P\in \mathcal M(Y^{\mathbb N})$.
 \end{maintheorem}

\subsection{Katok's entropy}

In \cite{CRVIII} the authors considered an extension of the Katok's entropy when the dynamical systems under consideration is a free semigroup action. 
\begin{definition}\label{def:metric-entropy-5}
Given probability measure $\mathbb P$ on $Y^\mathbb N$ and a Borel probability measure $\nu$ on $X$, $\delta\in (0,1)$ and $\varepsilon>0$, define
\begin{equation}
h^{K}_{\nu}(\mathbb{S},\varepsilon,\delta) =  \limsup_{n\to \infty} \,\,\frac1n \log \int_{\Sigma_p^+}\, s_\nu(g_{\omega_n} \dots g_{\omega_1}, n, \varepsilon, \delta)\; d\mathbb P(\omega)
\end{equation}
where $\omega = \omega_1 \,\omega_2 \cdots \omega_n \cdots$,
$$s_\nu(g_{\omega_n} \dots g_{\omega_1}, n, \varepsilon,\delta) = \,\inf_{\{E \,\subseteq \,X\,\colon\, \nu(E) \,> \,1-\delta\}}\,\, s(g_{\omega_n} \dots g_{\omega_1}, n, \varepsilon,E)$$
and $s(g_{\omega_n} \dots g_{\omega_1}, n, \varepsilon,E)$ denotes the maximal cardinality of the $(g_{\omega_n} \dots g_{\omega_1}, n, \varepsilon)$-separated subsets of $E$.
\end{definition}
 The \emph{entropy of the semigroup action $\mathbb{S}$ with respect to $\nu$ and $\mathbb P$} is defined by
\begin{equation}
h^{K}_{\nu}(\mathbb{S},\mathbb P) = \lim_{\delta \to 0} \,\,\lim_{\varepsilon \to 0}\,\, \limsup_{n\to \infty} \,\,\frac1n \log \int_{\Sigma_p^+}\, s_\nu(g_{\omega_n} \dots g_{\omega_1}, n, \varepsilon, \delta)\; d\mathbb P(\omega)
\end{equation}

Observe that the previous limit is well defined due to the monotonicity of the function
$$(\vep,\delta) \mapsto \frac1n \log \int_{\Sigma_p^+}\, s_\nu(g_{\omega_n} \dots g_{\omega_1}, n, \varepsilon, \delta)\; d\mathbb P(\omega)$$
on the unknowns $\varepsilon$ and $\delta$. Moreover, if the set of generators is $G_1=\{Id,f\}$, we recover the notion proposed by Katok for a single dynamics $f$.

 In  \cite{VV} the authors proved that for a  compact metric space $(X,d)$ and a continuous map $f:X\to X$ holds the following variational principle 
for the metric mean dimension
$$
\displaystyle\overline{\text{mdim}}_M \,\Big(X,f, d, \mathbb{P}\Big)=\lim_{\delta\to0}\limsup_{\varepsilon\,\to\,0^+}\,\frac{\displaystyle\sup_{\nu\in\mathcal M(X)}h_\nu^K(X,f,\varepsilon,\delta)}{-\log\varepsilon}, \\
$$
which, in the case where the dynamical systems is given by a free semigroup action, may be extend as:
  \begin{maintheorem}\label{thmB}
 Let $(X,d)$ be a compact metric space and $\mathbb{S}$ be the free semigroup action induced on $(X,d)$ by a family of continuous maps $(g_y \colon X \to X)_{y \,\in\, Y}$. Then
 $$
\displaystyle\overline{\text{mdim}}_M \,\Big(X,\mathbb S, d, \mathbb{P}\Big)\geq\lim_{\delta\to0}\limsup_{\varepsilon\,\to\,0^+}\,\frac{\displaystyle\sup_{\nu\in\mathcal M(X)}h_\nu^K(\mathbb S,\mathbb P,\varepsilon,\delta)}{-\log\varepsilon}, \\
$$
for every $\mathbb P\in\mathcal M(Y^\mathbb N)$. If $\mathbb P=\gamma^\mathbb N$, with $\gamma$ an homogeneous probability measure on $Y$, then 

$$
\displaystyle\overline{\text{mdim}}_M \,\Big(X,\mathbb S, d, \mathbb{P}\Big) = \lim_{\delta\to0}\limsup_{\varepsilon\,\to\,0^+}\,\frac{\displaystyle\sup_{\nu\in\mathcal M(X)}h_\nu^K(\mathbb S,\mathbb P,\varepsilon,\delta)}{-\log\varepsilon}.
$$
 \end{maintheorem}

\medskip

\subsection{Entropy of an open cover for a free semigroup action}

Consider   $\mathbb P\in\mathcal M(Y^{\mathbb N})$. Let $\mathcal U=\{U_1,\dots,U_k\}$ be a finite open cover of $X$. For  each $\omega \in Y^{\mathbb N}$ and 
$n\in\mathbb N$ define 
$$
\mathcal U(\omega,n)=\left\{U_{i_0}\cap (f_\omega^{1})^{-1}(U_{i_1})\cap\dots\cap(f_\omega^{n-1})^{-1}(U_{i_{n-1}}):U_{i_j}\in \mathcal U\right\}.
$$
Let $N_\nu(\mathcal U, w,n)$ is the minimal cardinal of a subcover of $\mathcal{U}(w,n)$. Finally, define
$$
h_{top}(\mathcal U, \mathbb S,\mathbb P)=\displaystyle\limsup_{n\,\to\,\infty}\,\displaystyle\frac{1}{n}\,\log\displaystyle\int_{Y^{\mathbb N}}\,N(\mathcal U,\omega,n)\,\,d\mathbb{P}(\omega).
$$
As a consequence of \cite[Theorem 2.4]{TBW} we have that
$$
h_{top}(Y^{\mathbb N}\times X,\mathbb S,\mathbb P)=\sup_{\mathcal U}h_{top}(\mathcal U, \mathbb S,\mathbb P),
$$
where the open covers under consideration in the above supremum  are those which are finite and with finite topological entropy.
\subsection{Shapira's entropy of a semigroup action}\label{subsection:shapira_entropy}
 For $\nu\in \mathcal{M}(X)$, for   $\delta\in(0,1)$ let $N_\nu(\mathcal U, w,n,\delta)$ the minimal cardinal of a subcover of $\mathcal{U}(w,n)$, up to a set  of $\nu$-measure less than $\delta>0$. Define 

 \begin{equation}\label{eq:defhep-s}
h^{\mathcal S}_\nu(\mathcal U,\mathbb S,\mathbb{P})=\lim_{\delta\to0}\displaystyle\limsup_{n\,\to\,\infty}\,\displaystyle\frac{1}{n}\,\log\displaystyle\int_{Y^{\mathbb N}}\,N_\nu(\mathcal U,\omega, n, \delta)\,\,d\mathbb{P}(\omega).
\end{equation}
We call $h_{\nu}(\mathcal U,\mathbb S,\mathbb{P})$ the \emph{metric entropy of the cover} $\mathcal U$ with respect to $\nu$. 
As 
$$N_\nu(\mathcal U,\omega, n, \delta)\leq N(\mathcal U,\omega, n) \text{ for every } \delta\in (0,1),$$ 
we have that
$h^{\mathcal{S}}_\nu(\mathcal U,\mathbb S,\mathbb{P})\leq h_{top}(\mathcal U,\mathbb S,\mathbb{P})$. 
It is important to mention that when $G_{1}=\{f\}$, our definition coincides with the classical one given in \cite{SH}.

 Before we state our theorem we need to introduce some notation. Associated to an open cover $\mathcal U$ of $X$, let $\Tilde{\mathcal U}=\{[i]\times V: i=1,\dots,p \text{ and } V\in \mathcal U\}$ and, for $\nu\in\mathcal{M}(X)$, denote  $\Pi(\sigma,\nu)_{erg}$  the set of $T_G$-invariant measures 
 so that the marginal in $\Sigma_p^+$ is $\sigma$-invariant and $\nu$ is the marginal in $X$.

\begin{maintheorem}\label{prop-1}
Let $(X,d)$ be a compact metric space and $\mathbb{S}$ be the free semigroup action induced on $(X,d)$ by a finite  family of continuous maps $(g_i \colon X \to X)_{i=1}^p$. 
 Under the above conditions we have that\\
(a) $ h_{top}(\mathcal U,\mathbb S,\eta_{\underline p} )= h_{top}(\Tilde{\mathcal U},T_G)-\log p;$\\
(b) $h_{top}( \mathcal U,\mathbb S,\eta_{\underline p})=\displaystyle\sup \left\{ h^S_{\nu}( \mathcal U,\mathbb S,\eta_p): \nu\in \mathcal M(X)\text{ and }\Pi(\sigma,\nu)_{erg}\not=\emptyset\right\}$,\\
where $\eta_{\underline p}=\left(\frac{1}{p},\dots,\frac{1}{p}\right)^{\mathbb N}$.
 \end{maintheorem} 
As a direct consequence of Theorem \ref{prop-1} and \cite{CRVIII} we have the following.
\begin{maincorollary}
 Let $(X,d)$ be a compact metric space and $\mathbb{S}$ be the free semigroup action induced on $(X,d)$ by a finite  family of continuous maps $(g_i \colon X \to X)_{i=1}^p$.  Then
\begin{align*}
  h_{top}(\mathbb S,\eta_{\underline p} )&=\sup_{\mathcal U} \displaystyle\sup_{\{\nu\in \mathcal M(X)\text{ and }\Pi(\sigma,\nu)_{erg}\not=\emptyset\}} h^S_{\nu}( \mathcal U,\mathbb S,\eta_{\underline p})) \\
                            &=h_{top}(F_G)-\log p,
\end{align*}
where $\eta_{\underline p}=\left(\frac{1}{p},\dots,\frac{1}{p}\right)^{\mathbb N}$.
\end{maincorollary}

In \cite{RX} it was proved that, for a  compact metric space $(X,d)$ and a continuous map $f:X\to X$,

$$
\displaystyle\overline{\text{mdim}}_M \,\Big(X,f, d\Big) = \limsup_{\varepsilon\,\to\,0^+}\,\frac{\displaystyle\sup_{\nu\in\mathcal M(X)}\inf_{\text{diam}(\mathcal U)\leq\varepsilon}h_\nu^S(\mathcal U,f)}{-\log\varepsilon}.
$$
In the next theorem we extend such result to the semigroup setting.
 \begin{maintheorem}\label{thmA}
 Let $(X,d)$ be a compact metric space and $\mathbb{S}$ be the free semigroup action induced on $(X,d)$ by a family of continuous maps $(g_y \colon X \to X)_{y \,\in\, Y}$. If $\mathbb P=\gamma^\mathbb N$ and $\gamma\in \mathcal M(Y)$ is homogeneous, then
 $$
\displaystyle\overline{\text{mdim}}_M \,\Big(X,\mathbb S, d, \mathbb{P}\Big) = \limsup_{\varepsilon\,\to\,0^+}\,\frac{\displaystyle\sup_{\{\nu\in\mathcal M(X):\Pi(\sigma,\nu)_{erg}\not=\emptyset\}}\inf_{\text{diam}(\mathcal U)\leq\varepsilon }h_\nu^S(\mathbb S,\mathcal U)}{-\log\varepsilon}.
$$
 \end{maintheorem}

\subsection{Ghys-Langevan-Walczack entropy} 
 Ghys, Langevin and Walczak proposed in \cite{GLW} the following definition of topological entropy of  a semigroup action given by a  finitely generated . A subset $E$ of a compact metric space $(X,d_X)$ is \emph{$(n,\vep)$-separated points by elements of $\mathbb S$} if for any $x\neq y$ in $E$ there exists $0\leqslant j \leqslant n$ and $g\in G_j$ such that $d(g(x), \,g(y)) > \vep$. The topological entropy of the semigroup action $\mathbb S$, induced by a semigroup $G$ generated by a finite set $G_1$ of continuous maps, is given by
\begin{equation}\label{def:GLW}
h_{GLW}(\mathbb S) = \lim_{\vep\,\to\, 0^+}\,\limsup_{n\,\to\,+\infty}\, \frac1n \log s(n,\vep)
\end{equation}
where $s(n,\vep)$ is the largest cardinality of $(n,\vep)$-separated points by elements of $\mathbb S$. Observe that, since $X$ is compact, $s(n, \varepsilon)$ is finite for every $n \in \mathbb{N}$ and $\varepsilon > 0$. Moreover, the map
$$\varepsilon > 0 \quad \mapsto \quad h_{GLW}(\mathbb S,\varepsilon)=\limsup_{n\, \to\, +\infty}\,\frac{1}{n}\,\log \,s(n,\varepsilon)$$
is monotonic, so $h_{GLW}(\mathbb S)$ is well defined (though it depends on the set $G_1$ of generators). This is a purely topological notion, independent of any previously fixed random walk on the semigroup. Observe also that
$$\sup_{g \,\in \,G_1}\, \quad h_{\text{top}}(g) \,\,\leqslant \,\, h_{GLW}(\mathbb S)$$
but this inequality may be strict (cf. \cite{GLW}).
\subsubsection{Metric mean dimension in the GLW setting}
As a natural extension of the metric mean dimension for a single dynamics we can consider the \emph{ upper GLW-metric mean dimension}
as 
\begin{equation}\label{GLW-mdim}
    \displaystyle\overline{\text{mdim}}_M^{GLW} \,\Big(X,\mathbb S, d\Big)=\limsup_{\varepsilon\to0}\frac{h_{GLW}(\mathbb S,\varepsilon)}{-\log\varepsilon}.
\end{equation}

As a direct consequence of the above definition we have that for any $\mathbb P\in \mathcal M(Y)$,
$$
 \displaystyle\overline{\text{mdim}}_M \,\Big(X,\mathbb S,\mathbb P, d\Big)\leq \displaystyle\overline{\text{mdim}}_M^{GLW} \,\Big(X,\mathbb S, d\Big),
$$
and in the case where the generating setting consists of a single dynamics the two definitions coincide with the classical one.
\subsubsection{Local measure entropy and measure metric mean dimension}
For $n\in\mathbb N$ let
$$
B_n^G(x,\varepsilon)=\{y\in X: d(g(x),g(y))<\varepsilon \text{ for all }g\in G_j,\;0\leq j\leq n\}
$$
the dynamical ball of center $x$, radius $\varepsilon$ and depth $n$. For any $\nu\in\mathcal{M}(X)$ the quantity
$$
h_{\nu}^G(x)=\lim_{\varepsilon\to0}h_{\nu}^G(x,\varepsilon)
$$
where 
$$
h_{\nu}^G(x,\varepsilon)=\limsup_{n\to\infty}-\frac{1}{n}\log\nu(B_n^G(x,\varepsilon))
$$
is called the \emph{local upper $\nu$-measure entropy} at the point $x$. If one takes $\liminf$ with respect to $n$ in 
the above definition we the \emph{local lower $\nu$-measure entropy} at the point $x$, denoted by $h_{\nu,G}(x)$. 
These quantities were defined and explored in \cite{Bis}, where the author proved that in the case of $\nu$ being a 
$G$-homogeneous measure $h_{\nu}^G(x)=h_{GLW}(\mathbb S)$, for all $x\in X$ (see  Section \ref{section3} for the definition of $G$-homogeneous measure).

In order to have a concept related to the metric mean dimension we define the  \emph{local upper measure metric mean dimension} as 
\begin{equation}
    \displaystyle\overline{\text{mdim}}_\nu \,\Big(X,\mathbb S, d\Big)=\limsup_{\varepsilon\to0}\frac{h_{\nu}^G(x,\varepsilon)}{-\log\varepsilon}.
\end{equation}
If one takes $\liminf$ in $\varepsilon$ we have the \emph{lower  local upper  measure metric mean dimension}, denoted by $\displaystyle\underline{\text{mdim}}_\nu \,\Big(X,\mathbb S, d\Big)$.

If, instead of $h_{\nu}^G(x,\varepsilon)$ we consider $h_{\nu,G}(x,\varepsilon)$
we have the \emph{upper  local lower  measure metric mean dimension} and \emph{lower  local lower  measure metric mean dimension}, denoted by 
$\displaystyle\overline{\text{mdim}}'_\nu \,\Big(X,\mathbb S, d\Big)$ and $\displaystyle\underline{\text{mdim}}'_\nu \,\Big(X,\mathbb S, d\Big)$, respectively.
\begin{remark}
All the above definitions could be made in terms of dynamical balls. 
\end{remark}

In the case the where the ambient space $X$ is an oriented manifold it admits a volume form $dV$ which induces a
natural volume measure $\nu_v$ on the Borel sets defined as 
$$
\nu_v(A)=\int_A \; dV.
$$
The next gives  a kind of partial variational principle for the metric mean dimension of the group action in  terms of the volume
measure.
\begin{maintheorem}\label{thm8}
Let $(G,G_1)$ be a finitely generated group of homeomorphisms of a compact closed and oriented manifold $(M,d)$. Let $s\in(0,\infty)$
and $\nu_v$ the natural volume on $M$. If 
$$
\overline{\text{mdim}}_{\nu_v} \,(x, d)\leq s \text{ for all }x\in M \text{ then }\displaystyle\overline{\text{mdim}}_M^{GLW} \,\Big(X,\mathbb S, d\Big)\leq s.
$$
\end{maintheorem}
Our last theorem shows that, in the case where the group action admits a strongly  $G$-homogeneous measure $\nu$  we have an equality between the local measure   metric measure mean dimension of $\nu$ and the metric mean dimension of the group action (see Section \ref{section3} for the definition of strongly  $G$-homogeneous measure).
\begin{maintheorem}\label{thm7}
Let $(X,d)$ be a compact metric space and $\mathbb{S}$ be the semigroup action induced on $(X,d)$ by a finite  family of continuous maps $(g_i \colon X \to X)_{i=1}^p$. \\
(a) If  $\nu\in\mathcal M(X)$ is strongly $G$-homogeneous then
$$\displaystyle\overline{\text{mdim}}_M^{GLW} \,\Big(X,\mathbb S, d\Big)=\limsup_{\varepsilon\to0}\frac{h_{\nu}^G(x,\varepsilon
)}{-\log\varepsilon}.
$$
(b) Let $\nu$ be a Borel measure on $X$ and  $s\in(0,\infty)$. If 
$$\inf_{x\in X}\displaystyle\overline{\text{mdim}}'_\nu \,(x, d)\geq s \text{\;\; then \;\;}\displaystyle\overline{\text{mdim}}_M^{GLW} \,\Big(X,\mathbb S, d\Big)\geq s.$$
\end{maintheorem}
 \section{Some facts about homogeneous measures and  $G$-homogeneous measures }\label{section3}
In order to obtain a text as self-contained as possible, in this section we recall the definitions of upper box dimension, homogeneous measure and $G$-homogeneous measure.
\subsection{Upper box dimension}

Let $(Y,d_Y)$ be a compact metric space.

\begin{definition}\label{def.ubd-Y}
The \emph{upper box dimension} of $(Y,d_Y)$ is given by
\begin{equation}\label{eq:boxd}
\overline{\text{dim}}_B Y=\limsup_{\varepsilon\,\to\,0^+}\,\frac{\log N(\varepsilon)}{|\log\varepsilon|},
\end{equation}
where $N(\varepsilon)$ stands for the maximal cardinality of an $\varepsilon$--separated set in $(Y,d_Y)$.
\end{definition}


Consider now a Borel probability measure $\nu$ on $Y$.
\begin{definition}\label{def.ubd-niu}
The \emph{upper box dimension} of $\nu$ is given by
$$\overline{\dim}_B \, \nu = \lim_{\delta \,\to \,0^+} \,\inf \,\Big\{\overline{\dim}_B \,Z \colon \,\, Z\subset Y \quad \text{and} \quad \nu(Z) \geqslant 1-\delta\Big\}.$$
\end{definition}

It is worth mentioning that, although the upper box dimension of a set $Z$ coincides with the upper box dimension of its closure, the upper box dimension of a probability measure is intended to estimate the size of subsets rather than the entire support of the measure (that is, the smallest closed subset with full measure). Indeed, it may happen that $\overline{\dim}_B \, \nu < \overline{\dim}_B \, (\supp \,\nu)$ (cf. Example 7.1 in \cite{Pesin}). We refer the reader to \cite{F90, Pesin} for excellent accounts on dimension theory.

\subsection{Homogeneous measures} Let $\nu$ be a Borel probability measure on the compact metric space $(Y,d_Y)$. A
balanced measure should give the same probability to any two balls with the same radius, but this is in general a too strong demanding. Instead, we weaken the request in the following way.

\begin{definition}\label{def.hom}
We say that $\nu$ is \emph{homogeneous} if there exists $L>0$ such that
\begin{equation}\label{def:homogeneous}
\nu\big(B(y_1,2\vep)\big) \leqslant L \, \nu\big(B(y_2,\vep)\big)\quad\quad \forall\, y_1, \,y_2 \,\in \,\supp \,\nu \quad \forall\, \vep>0.
\end{equation}
\end{definition}

For instance, the Lebesgue measure on $[0,1]$, atomic measures and probability measures absolutely continuous with respect to the latter ones, with densities bounded away from zero and infinity, are examples of homogeneous probability measures. We denote by $\mathcal{H}_Y$ the set of such homogeneous Borel probability measures on $Y$.

By definition, every homogeneous measure satisfies
\begin{equation}\label{def:diamreg}
\nu\big(B(y, 2 \vep)\big) \leqslant L \, \nu\big(B(y,\vep)\big)\quad\quad \forall\, y \,\in \,\supp \,\nu \quad \forall\, \vep>0
\end{equation}
and, as $\nu\big(B(y_1,\vep)\big) \leqslant \nu\big(B(y_1,2 \vep)\big)$,
\begin{equation}\label{def:homogeneous0}
\nu\big(B(y_1,\vep)\big) \leqslant L \, \nu\big(B(y_2,\vep)\big)\quad\quad \forall\, y_1, \,y_2 \,\in \,\supp \,\nu \quad \forall\, \vep>0.
\end{equation}
A measure $\nu$ satisfying \eqref{def:diamreg} is said to be a \emph{doubling measure}.
Although the two concepts \eqref{def:diamreg} and \eqref{def:homogeneous0} are unrelated in general, if $Y$ is a subset of an Euclidean space $\mathbb R^k$ then any probability $\nu$ satisfying \eqref{def:homogeneous0} is a doubling measure. Indeed, as there is a constant $C_k$ such that $Leb(B(y, r)) = C_k \,r^k$ for every $y \in Y$ and every $r>0$, any ball $B(y, 2 \vep)$ can be covered by at most $2^k$
balls of radius $\vep$; we now apply \eqref{def:homogeneous}. For a discussion on conditions on $Y$ which ensure the existence of homogeneous measures and further relations between homogeneity and the doubling property we refer the reader to \cite[Section~4]{Bis} and references therein.

\subsection{$G$-homogeneous measures} 
 For a compactly generated semigroup by a continuous family $(g_y \colon X \to X)_{y\in Y}$ acting on  a metric space, we say that a Borel measure  $\nu\in\mathcal M(X)$ is \emph{
$G$-homogeneous} if 
\begin{enumerate}
    \item[(a)] $\nu(K)<\infty$, for any compact set $K\subset X$;
    \item[(b)] there exists $K_0\subset X$ such that $\nu(K_0)>0$;
    \item[(c)] for any $\varepsilon>0$ there exist $\delta(\varepsilon)>0$ and $c>0$ such that
    $$
    \nu(B_n^G(x,\delta(\varepsilon)))\leq c\cdot\nu(B_n^G(y,\varepsilon))
    $$
    holds for any $n\in\mathbb N$ and all $x,y\in X$.
In the case where $\delta(\varepsilon)=O(\varepsilon)$ we say that $\nu$ is \emph{strongly $G$-homogeneous}.
\end{enumerate}

As examples of spaces which admit a strongly $G$-homogeneous measure we have the following:

\noindent \textbf{1}. The canonical volume form $dV$ on a closed, compact and oriented Riemannian manifold $X$ determines a  
strongly $G$-homogeneous measure $\nu$ if $G$ is a finitely generated group of isometries.

\noindent \textbf{2}. If $X$ is a locally compact topological group, $\mu$ is a right invariant measure and $G$ is a finitely generated group by $G_1=\{id_X,T_1,T_1^{-1},T_2,T_2^{-1},\dots,T_p,T_p^{-1}\}$, a finite and symmetric set  of homeomorphisms, then $\mu$ is strongly $G$-homogeneous (see \cite[Proposition 4.6]{Bis}).

\section{Proofs}\label{section4}
In this section we prove our main results.

\subsection{Proof of Theorem \ref{thm6}}
It is clear from the definition of the entropy function that
$h_d(x,\varepsilon)\leq h(X,\mathbb S,\mathbb P,\varepsilon)$, for all $x\in X$, and it implies that 
$$
\displaystyle\overline{\text{mdim}}_M \,\Big(X,\mathbb S, d, \mathbb{P}\Big)\geq \limsup_{\varepsilon\,\to\,0^+}\,\frac{\displaystyle\sup_{x\in X}h_d(x,\varepsilon)}{-\log\varepsilon}.
$$
To prove the converse inequality we start noticing that, for a fixed $\varepsilon>0$, if $X=\cup_{i=1}^kF_i$, finite union of closed sets, then $B(X,\mathbb S,\varepsilon,\mathbb P)\leq \max_{i}B(F_i,\mathbb S,\varepsilon,\mathbb P)$. Then cover $X$ by closed balls of radius 1, say
$\mathcal B_1=\{B_1^1,\dots,B_{\ell_1}^1\}$ such cover. Let $B_{j_1}^1$ be the closed ball in the given cover where the maximum occurs. Now cover 
$B_{j_1}^1$ by a finite family of  closed balls of radius at most $\frac{1}{2}$ denoted by $\{B_1^2,\dots,B_{\ell_2}^2\}$. Again
there exists $B_{j_2}^2\in \mathcal B_2$ for which $B(X,\mathbb S,\varepsilon,\mathbb P)\leq B(B_{j_2}^2,\mathbb S,\varepsilon,\mathbb P)$. Follwoing by induction, for each $k\in\mathbb N$, there exists a closed ball of radius at most $\frac{1}{k}$ so that $B(X,\mathbb S,\varepsilon,\mathbb P)\leq B(B_{j_k}^k,\mathbb S,\varepsilon,\mathbb P)$. Moreover, by the previous construction we have a sequence of nested closed balls $\{B_{j_k}^k\}_{k\in\mathbb N}$ whose diameter goes to zero. So, there exists $\bar x= \cap_{k\in\mathbb N} B_{j_k}^k$ and for
any closed neighbourhood $F$ of $\bar x$ we have $B_{j_k}^x\subset F$, for $k\in\mathbb N$ large enough. It gives 
$$
h_d(\bar x,\varepsilon)\geq B(F,\mathbb S,\varepsilon,\mathbb P)\geq B(B_{j_k}^k,\mathbb S,\varepsilon,\mathbb P)\geq B(X,\mathbb S,\varepsilon,\mathbb P).
$$
Hence 
\begin{align*}
    \limsup_{\varepsilon\,\to\,0^+}\,\frac{\displaystyle\sup_{x\in X}h_d(x,\varepsilon)}{-\log\varepsilon}\geq \displaystyle\overline{\text{mdim}}_M \,\Big(X,\mathbb S, d, \mathbb{P}\Big)
\end{align*}
and it finishes the proof.




\subsection{Proof of Theorem \ref{thmB}}
First we notice that for any $\nu\in\mathcal{M}(X)$ and $\delta>0$, $\nu(X)>1-\delta$ and so, for every $\varepsilon>0$, $n\in\mathbb N$ and $\omega\in Y^\mathbb N$
$$
s(g_{\omega_n}\dots g_{\omega_1},n,\varepsilon)\geq s_\nu(g_{\omega_n}\dots g_{\omega_1},n,\varepsilon,\delta).
$$
It implies that, for  any $\nu\in\mathcal{M}(X)$
$$
h(X,\mathbb S,\mathbb P,\varepsilon)\geq h^K_\nu(X,\mathbb S,\mathbb P,\varepsilon,\delta).
$$
Hence,

\begin{align}\label{eq:20}
\displaystyle\overline{\text{mdim}}_M \,\Big(X,\mathbb S, d, \mathbb{P}\Big)\geq \lim_{\delta\to0}\limsup_{\varepsilon\,\to\,0^+}\,\frac{\displaystyle\sup_{\nu\in\mathcal M(X)}h_\nu^K(\mathbb S,\mathbb P,\varepsilon,\delta)}{-\log\varepsilon}.
\end{align}
If $\mathbb P=\gamma^\mathbb N$ with  $\gamma\in \mathcal{H}_Y$, by \eqref{ineq:N-S} we know that for $\nu\in\mathcal{M}(X)$ and 
$\mu\in\Pi(\sigma,\nu)\not=\emptyset$,

$$
h_\nu^K(X,\mathbb S,\mathbb P,\varepsilon,\delta)\geq \displaystyle\sup_{\mu\in\Pi(\sigma,\nu)} h_\mu^K(Y^\mathbb N\times X,T_G,\varepsilon,\delta)-\log N_Z(\varepsilon).
$$
It follows that
\begin{align*}
    \lim_{\delta\to0}\limsup_{\varepsilon\,\to\,0^+}\,\frac{\displaystyle\sup_{\nu\in\mathcal M(X)}h_\nu^K(\mathbb S,\mathbb P,\varepsilon,\delta)}{-\log\varepsilon}
    &\geq \lim_{\delta\to0}\limsup_{\varepsilon\,\to\,0^+}\,\frac{\displaystyle\sup_{\mu\in\mathcal M_{T_G}(Y^\mathbb N\times X)}h_\nu^K(T_G,\varepsilon,\delta)}{-\log\varepsilon}-\displaystyle\overline{\text{dim}}_B(\text{supp}(\gamma))\\
    &=\displaystyle\overline{\text{mdim}}_M \,\Big(Y^\mathbb N\times X,T_G, D\times d\Big)-\displaystyle\overline{\text{dim}}_B(\text{supp}(\gamma))\\
    &=\displaystyle\overline{\text{mdim}}_M \,\Big(X,\mathbb S, d, \mathbb{P}\Big).
\end{align*}
By \eqref{eq:20} we have the desired equality  and conclude the proof.

\subsection{Proof of Theorem \ref{prop-1}} Take $i_0,\dots,i_{n-1}\in \{1,\dots,p\}$, $U_{j_0},\dots,U_{j_{n-1}}\in \mathcal U$ and consider
 \begin{align*}
     &\left([i_0]\times U_{j_0}\right) \cap\left( T_G^{-1}([i_1]\times U_{j_1}\right)\cap\dots\cap T_G^{-1}\left([i_{n-1}]\times U_{j_{n-1}}\right)\\
     &=[i_0\dots i_{n-1}]\times \left( U_{j_0}\cap \dots \cap (f_\omega^{n-1})^{-1}(U_{j_{n-1}})\right),
 \end{align*}
 where $\omega $ belongs to the cylinder set $[i_0\dots i_{n-1}]$. If we denote by $\mathcal U(\omega,n)=\{V_{j_0}\cap \dots \cap (f_\omega^{n-1})^{-1}(V_{j_{n-1}}): V_{j_\ell}\in\mathcal U\}$ the open cover of $X$ induced by $\omega$, we have that
 $N(\mathcal U,\omega,n)$ coincides with the minimum  number of open sets of $\tilde{\mathcal U}^{(n)}$ necessary to cover $[i_0\dots i_{n-1}]\times X$. So,
 \begin{align*}
    h_{top}(\mathcal U, \mathbb S,\eta_{\underline p})+\log p&=
    \lim_{n\to\infty} \frac{1}{n}\log\left(\frac{1}{p^n}\sum_{\underline g\in G_n}N(\mathcal U, \underline g,n)\right)+\log p\\
    &= \lim_{n\to\infty}  \frac{1}{n}\log N(\Tilde{\mathcal U},T_G,n)\\
    &=h_{top}(\Tilde{\mathcal U},T_G).
 \end{align*}
It proves item (i).

For the second item take $\delta\in (0,1)$ and  $\nu\in\mathcal M(X)$ so that $\Pi(\sigma,\nu)_{erg}\not=\emptyset$. For $\mu\in \Pi(\sigma,\nu)_{erg}$ we have that
\begin{align*}
    \sum_{\underline g\in G_n}N_{\nu}(\mathcal U,\underline g,n,\delta)= N_\mu(\mathcal U,T_G,n,\delta).
\end{align*}
The equality comes from the fact that if 
\begin{align*}
    \sum_{\underline g\in G_n}N_{\nu}(\mathcal U, \underline g,n,\delta)> N_\mu(\mathcal U,T_G,n,\delta),
\end{align*}
there exists a cylinder $[i_0\dots i_{n-1}]$ so that $[i_0\dots i_{n-1}]\times X$ is covered by at most $N_{\nu}(\mathcal U,\underline g,n,\delta)-1$
open sets, where $w=i_0\dots i_{n-1}$. As $(\pi_X)_*(\mu)=\nu$, it contradicts the minimality of $N_{\nu}(\mathcal U,\underline g,n,\delta)$.
So, 
\begin{align*}
  &\displaystyle\sup_{ \left\{ \nu\in \mathcal M(X)\text{ and }\Pi(\sigma,\nu)_{erg}\not=\emptyset\right\}} h^S_{\nu}( \mathcal U,\mathbb S)\\
  &=\displaystyle\sup_{ \left\{ \nu\in \mathcal M(X)\text{ and }\Pi(\sigma,\nu)_{erg}\not=\emptyset\right\}} \lim_{n\to\infty} \frac{1}{n}\log\left(\frac{1}{p^n}\sum_{\underline g\in G_n}N_\nu(\mathcal U, \underline g,n)\right)\\
  &=\sup_{\mu\in \mathcal E_{T_G}(\Sigma_p^+\times X)} \lim_{n\to\infty}  \frac{1}{n}\log N_\mu(\Tilde{\mathcal U},T_G,n)-\log p\\
  &=h_{top}(\Tilde{\mathcal U},T_G)-\log p\\
  &=h_{top}(\mathcal U,\mathbb S,\eta_{\underline p}),
 \end{align*}
which concludes the proof of the second item.

\subsection{Proof of Theorem \ref{thmA}}
Before we start the proof we observe that Definition \ref{def:metric-entropy-5} could be made in terms of spanning sets. More precisely, given $\varepsilon>0$, a positive integer $n$ and $\underline{g}=g_{\omega_n}\dots g_{\omega_1}$, we say that a subset $A$ of $E\subset X $ is a  $(g_{\omega_n}\dots g_{\omega_1},n,\varepsilon,E)-$spanning set if
  for any $x\in E$ there exists $y\in A$ so that $D_{\underline g}(x,y)<\varepsilon$. By the compactness of $X$, given $\varepsilon$, $n$ and $\underline{g}$ as before, there exists a finite $(\underline g,n,\varepsilon,E)-$spanning set.

  We denote by $b(g_{\omega_n}\dots g_{\omega_1},n, \varepsilon, E)$ the minimum cardinality of a $(g_{\omega_n}\dots g_{\omega_1},n,\varepsilon,E)$-spanning. For $\delta>0$ we set
  $$b_\nu(g_{\omega_n} \dots g_{\omega_1}, n, \varepsilon,\delta) = \,\inf_{\{E \,\subseteq \,X\,\colon\, \nu(E) \,> \,1-\delta\}}\,\, b(g_{\omega_n} \dots g_{\omega_1}, n, \varepsilon,E).$$
  It is not difficult to see that

\begin{equation}
\nonumber h^{K}_{\nu}(\mathbb{S}, \mathbb P) = \lim_{\delta \to 0} \,\,\lim_{\varepsilon \to 0}\,\, \limsup_{n\to \infty} \,\,\frac1n \log \int_{\Sigma_p^+}\, b_\nu(g_{\omega_n} \dots g_{\omega_1}, n, \varepsilon, \delta)\; d\mathbb P(\omega).
\end{equation}

Let us proceed to the proof of the theorem.
Fix $\varepsilon>0$ and consider a positive integer $k=k(\varepsilon)\geq 1$ so that 
$\sum_{i\geq k}\frac{diam(Y)}{2^i}<\frac{\varepsilon}{2}$. For $\gamma\in \mathcal H_Y$, take $Z=\text{supp}(\gamma)$ and choose a maximal $\frac{\varepsilon}{4}$-separated set $E\subset Z$, whose cardinality is denoted by $N_Z(\varepsilon
)$. By the definition of upper box dimension, 
$$
\limsup_{\varepsilon\to0}\frac{N_Z(\varepsilon)}{-\log\varepsilon}=\overline{dim}_B(Z).
$$
For each $n\in\mathbb N$ and  each point $(p_1,\dots,p_{n+k})\in E^{n+k}$, consider the cylinder 
$$
C_{i_1\dots i_{n+k}}=\left\{\omega\in Y^{\mathbb N}:\omega_i\in B\left(p_i,\frac{\varepsilon}{4}\right), \text{ for }i=1,\dots,n+k\right\}.
$$
Note that the collection of cylinders defined above covers $Z^\mathbb{N}$ and has diameter less than $\varepsilon$.

Now, for the fixed $\varepsilon$ let $\mathcal{U}_0$ be an open cover of $X$ with $\text{diam}(\mathcal{U}_0)\leq \varepsilon$ and  $Leb(\mathcal U_0)\geq \varepsilon$.
If  $\mathcal U$ is an open cover of $X$ with diameter less or equal to $\displaystyle\frac{\varepsilon}{8}$, $\omega\in Y^\mathbb N$, as $Leb(\mathcal U_0)\geq \text{diam}(\mathcal U)$, $\mathcal U(\omega,n)$  refines $\mathcal U_0(\omega,n)$. It implies that, for $\delta\in(0,1)$, $N_\nu(\mathbb S, \mathcal U, g_{\omega_n}\dots g_{\omega_1}, n,\delta)\geq N_\nu(\mathbb S, \mathcal U_0, g_{\omega_n}\dots g_{\omega_1}, n,\delta)$. Thus, once 
$$
N_\nu(\mathbb S,\mathcal U, g_{\omega_n}\dots g_{\omega_1},n,\delta)\geq s_\nu(\mathbb S, g_{\omega_n}\dots g_{\omega_1}, n,\varepsilon,\delta)\geq b_\nu(\mathbb S, g_{\omega_n}\dots g_{\omega_1}, n,\varepsilon,\delta), 
$$
$\text{for all }\omega\in Y^\mathbb N \text{ and } n\in \mathbb N$,
we have that

\begin{align}\label{eq:ub-1B}
 &   \int_{Y^{\mathbb{N}}}N_\nu(\mathbb S,\mathcal U,g_{\omega_n}\dots g_{\omega_1},n,\delta)\ d\mathbb P(\omega)
 \nonumber   \geq \int_{Y^{\mathbb{N}}}N_\nu(\mathbb S, \mathcal U_0,g_{\omega_n}\dots g_{\omega_1}, n,\delta)\ d\mathbb P\\
  \nonumber  &\geq \int_{Y^{\mathbb{N}}}s_\nu(\mathbb S, g_{\omega_n}\dots g_{\omega_1}, n,\varepsilon,\delta)\ d\mathbb P\\
     &\geq \int_{Y^{\mathbb{N}}}b_\nu(\mathbb S, g_{\omega_n}\dots g_{\omega_1}, n,\varepsilon,\delta)\ d\mathbb P\\
    \nonumber &\geq\sum_{\underline i=(i_1\dots i_{n+k})}\min_{\omega\in C_{\underline i}\cap Z^{\mathbb{N}}} b_\nu(\mathbb S, g_{\omega_n}\dots g_{\omega_1}, n,\varepsilon,\delta)\times\min_{\underline i}\mathbb P(C_{\underline i}\cap Z).
\end{align}
Now we notice that the image of $b_\nu(\mathbb S, \cdot, n,\varepsilon,\delta):C_{\underline i}\to \mathbb Z_+ $ has a minimum in $\mathbb Z_+$ and such minimum is attained by some $\omega^{(\underline i)}\in C_{\underline i}$. So,
 This together with \eqref{eq:ub-1B}, the fact that $\mathbb P$ is a product measure
and the homogeneity assumption on $\gamma$ imply that
\begin{align}\label{ineq:N-S}
&\int_{Y^{\mathbb{N}}}N_\nu(\mathbb S,\mathcal U,g_{\omega_n}\dots g_{\omega_1},n,\delta)\ d\mathbb P(\omega)\\ \nonumber
         &\geq\int_{Y^{\mathbb N}}\,b_\nu(g_{\omega_n}\dots g_{\omega_1},n,\varepsilon,\delta)\,d\mathbb P(\omega)\\
\nonumber& \geqslant \, \Big[\sum_{\underline i\,=\,(i_1,i_2,\dots,i_{n+K})}
		\min_{\omega \,\in \,C_{\underline i} \cap Z^\N} \,b_\nu(g_{\omega_n}\dots g_{\omega_1},n,\varepsilon,\delta)\Big]
		\times \min_{\underline i} \mathbb P(C_{\underline i} \cap Z^\N) \\
\nonumber& \geqslant \,\sum_{\underline i}\, b_\nu(g_{\omega^{(\underline i)}},n,\varepsilon,\delta)
            	\, \times \, \min_{\underline i}  \prod_{j=0}^{n+K-1} \gamma\Big(B(p_{i_j},\frac{\vep}{4}) \cap Z\Big)
 \\
\nonumber&\geqslant \, b_\mu(T_G\mid_{Z^\N \times X},n,\varepsilon,\delta) \; \left(\frac{1}{L^2}\right)^{n+K} \left(\frac1{N_Z(\vep)}\right)^{n+K}\\
\nonumber&\geqslant \, N_\mu(T_G\mid_{Z^\N \times X},\mathcal V_0,n,\delta) \; \left(\frac{1}{L^2}\right)^{n+K}\left(\frac1{N_Z(\vep)}\right)^{n+K}
\end{align}
where by  $g_{\omega^{(\underline i)}}$ we mean $g_{\omega^{(\underline i)}_n}\dots g_{\omega^{(\underline i)}_1}$ if $\omega^{(\underline i)}|_{[1,n]}=\omega^{(\underline i)}_1\dots \omega^{(\underline i)}_n$ and $\mu\in \Pi(\sigma,\nu)$, $\mathcal V_0$ is an open cover with $Leb(\mathcal V_0)\leq \varepsilon$ and $L>0$ is specified by the homogeneity of $\gamma$ and does not depend on neither $\vep$ nor $n$. Notice that the inequality
$$\sum_{\underline i}\, b_\nu(g_{\omega^{(\underline i)}},n,\varepsilon,\delta) \,\geqslant \,b_\mu(T_G\mid_{Z^\N \times X},n,\varepsilon,\delta)$$
is a consequence of the fact that, if $\{x_1^{(i)},\dots,x_{b(g_{\omega^{(i)}},n,\varepsilon)}\}$ is a $(g_{\omega^{(i)}},n,\varepsilon)$--spanning set for  a subset   $\overline Z\subset Z$, satisfying $\nu(\overline Z)\geq 1-\delta$,  with smallest cardinality, then
$$\bigcup_{\underline i} \,\Big\{\Big(\omega^{(\underline i)},x_1^{(\underline i)}\Big),\dots,\Big(\omega^{(\underline i)},x_{b(g_{\omega^{(i)}},n,\varepsilon)}^{(\underline i)}\Big)\Big\}$$
is a $(T_G,n,\varepsilon)$--spanning set for $Y^\mathbb N\times \overline Z$ and $\mu(Y^\mathbb N\times \overline Z)=\nu(\overline Z)\geq 1-\delta$.
Besides, the inequality
$$\min_{\underline i}  \prod_{j=0}^{n+K-1} \gamma\Big(B(p_{i_j},\frac{\vep}{4})\Big) \geqslant  \left(\frac{1}{L^2}\right)^{n+K} \left(\frac1{N_Z(\vep)}\right)^{n+K}$$
is due to the homogeneity of $\gamma$, which implies that, for every $q \in \supp\, \nu$, any $p_{i_j}$ and all $\underline i$,
$$\gamma\Big(B(p_{i_j},\vep)\Big) \geqslant \frac{1}{L} \,\gamma\Big(B(q,\vep)\Big) \quad \quad \forall \, \vep >0$$
and the fact that, as $\bigcup_{e \,\in \,E} \,B(e,\frac{\vep}{4}) = Z$,
$$1 = \gamma\left(\bigcup_{e \,\in \,E} \,B(e,\frac{\vep}{4})\right) \leqslant \sum_{e \,\in \,E} \,\gamma\Big(B(e,\frac{\vep}{4})\Big) \,\leqslant \,N_Z(\vep)\, L\, \gamma\Big(B(q,\frac{\vep}{4})\Big)$$
thus
$$\gamma\Big(B(q,\frac{\vep}{4})\Big) \geqslant \frac{1}{L}\, \frac{1}{N_Z(\vep)}.$$
Then we notice that, by \eqref{ineq:N-S}
\begin{align*}
    \sup_{\{\nu\in\mathcal{M}:\Pi(\sigma,\nu)_{erg}\not=\emptyset\}}h_\nu^S(\mathbb S,\varepsilon,\mathbb P)
    &\geq \sup_{\mu\in \mathcal E(T_G)}h^S_\mu(T_G,\mathcal V_0)-\log N_Z(\varepsilon)\\
    &=h_{\text{top}}(T_G,\mathcal V_0)-\log N_Z(\varepsilon)\\
    &\geq h(T_G,3\varepsilon)-\log N_Z(\varepsilon).
\end{align*}
Therefore,
\begin{align}\label{mdim-geq}
\limsup_{\varepsilon\to0} \frac{h^S(\mathbb S,\varepsilon,\mathbb P)}{-\log\varepsilon}
    & \geqslant \overline{\text{mdim}}_M\,\Big(Z^{\mathbb N} \times X,T_G,D\times d\Big) - \limsup_{\varepsilon\,\to\,0^+} \,\frac{\log N_Z(\vep)}{-\log\varepsilon} \\ \nonumber
    & = \overline{\text{mdim}}_M\,\Big(Z^{\mathbb N} \times X,T_G,D\times d\Big) - \overline{\dim}_B Z \\ \nonumber
    & = \overline{\text{mdim}}_M\,\Big((\supp \,\nu)^{\mathbb N} \times X,T_G,D\times d\Big) - \overline{\dim}_B\,(\supp\,\nu)\\ \nonumber
    & = \overline{\text{mdim}}_M\,\Big(X,\mathbb S, d, \mathbb P\Big) .
\end{align}

For the converse inequality we observe that

\begin{align*}
    &\int_{Y^{\mathbb{N}}}N_\nu(\mathbb S,\mathcal U_0, g_{\omega_n}\dots g_{\omega_1},n,\delta)\ d\mathbb P(\omega)
    \leq \int_{Y^{\mathbb{N}}}s_\nu(\mathbb S,Leb(\mathcal U_0), g_{\omega_n}\dots g_{\omega_1},n,\delta)\ d\mathbb P\\ 
  \nonumber  &\leq \sum_{\underline i=(i_1\dots i_{n+k})}\left[\max_{\omega\in C_{\underline i}\cap Z^{\mathbb{N}}} s_\nu(\mathbb S,Leb(\mathcal U_0), g_{\omega_n}\dots g_{\omega_1},n,\delta)\times\mathbb P(C_{\underline i})\right].
\end{align*}
Now we notice that the image of $s_\nu(\mathbb S,\cdot,n, Leb(\mathcal U_0)):C_{\underline i}\to \mathbb Z_+ $ is contained in  $[0,s(T_G,n,Leb(\mathcal U_0))$. So, it has a maximum  in $\mathbb Z_+$ and such maximum  is attained by some $\omega^{(\underline i)}\in C_{\underline i}$. So, using the fact that
$$
N_\nu(\mathbb S,\mathcal U_0, g_{\omega_n}\dots g_{\omega_1},n,\delta)\leq s_\nu(\mathbb S, g_{\omega_n}\dots g_{\omega_1},n,Leb(\mathcal U_0),\delta)
$$
and 

$$
\sum_{\underline i=(i_1\dots i_{n+k})}s_\nu(\mathbb S, g_{\omega_n}\dots g_{\omega_1},n,Leb(\mathcal U_0),\delta) \leq s_\mu(T_G,n,Leb(\mathcal U_0))
$$
and that $\gamma $ is homegeneous  we obtain
\begin{align*}
    &\int_{Y^{\mathbb{N}}}N_\nu(\mathbb S,\mathcal U_0, g_{\omega_n}\dots g_{\omega_1},n,\delta)\ d\mathbb P(\omega)\\
    &\leq \int_{Y^{\mathbb{N}}}s_\nu(\mathbb S, g_{\omega_n}\dots g_{\omega_1},n,Leb(\mathcal U_0),\delta)\ d\mathbb P\\
  \nonumber  &\leq \sum_{\underline i=(i_1\dots i_{n+k})}\left[s_\nu(\mathbb S, g_{\omega_n}\dots g_{\omega_1},n,Leb(\mathcal U_0),\delta)\times\mathbb \displaystyle\max_{\omega\in C_{\underline i}\cap Z^{\mathbb{N}}} P(C_{\underline i})\right]\\\nonumber
  &\leq s_\mu(T_G,n,Leb(\mathcal U_0))\left(\frac{1}{N_Z(Leb(\mathcal U_0))}\right)^{n+K}\\ \nonumber
  &\leq N_\mu(T_G,\mathcal V_0,n)\left(\frac{1}{N_Z(Leb(\mathcal U_0))}\right)^{n+K},
\end{align*}
where $\mathcal V_0$ is a finite collection of open sets which covers $Y^\mathbb N\times X$ up to a set of $\mu$-measure less than $\delta$, $\displaystyle\frac{\varepsilon}{4}\leq Leb(\mathcal{U}_0)=\text{diam}(\mathcal{V}_0)\leq \varepsilon$ and $Leb(\mathcal V_0)\geq \displaystyle\frac{\varepsilon}{8}$. 
\begin{align*}
   h^S(\mathbb S,\varepsilon,\mathbb P)
   &=\sup_{\{\nu\in \mathcal M(X):\Pi(\sigma,\nu)_{erg}\not=\emptyset\}}\inf_{\text{diam}(\mathcal U)\leq \varepsilon} h_\nu^S(\mathbb S,\mathcal U,\mathbb P)\\
   &\leq\sup_{\{\nu\in \mathcal M(X):\Pi(\sigma,\nu)_{erg}\not=\emptyset\}}h_\nu^S(\mathbb S,\mathcal U_0,\mathbb P)\\
   &=\sup_{\{\nu\in \mathcal M(X):\Pi(\sigma,\nu)_{erg}\not=\emptyset\}}
    \limsup_{n\to\infty}\frac{1}{n}\log \int_{Y^{\mathbb{N}}}N_\nu(\mathbb S,\mathcal U_0, g_{\omega_n}\dots g_{\omega_1},n,\delta)\ d\mathbb P(\omega)\\
      &\leq\sup_{\{\nu\in \mathcal M(X):\Pi(\sigma,\nu)_{erg}\not=\emptyset\}}\sup_{\mu\in\Pi(\sigma,\nu)}
    \limsup_{n\to\infty}\frac{1}{n} \log N_\mu(T_G,\mathcal V_0,n,\delta)-\log N_Z(Leb(\mathcal U_0))\\
    &=h_{top}(T_G, \mathcal V_0)-\log N_Z\left(\frac{\varepsilon}{4}\right)\\
    &\leq  \limsup_{n\to\infty}\frac{1}{n} \log s(T_G,n,Leb(\mathcal V_0))-\log N_Z\left(\frac{\varepsilon}{4}\right)\\
    &\leq \limsup_{n\to\infty}\frac{1}{n} \log s\left(T_G,n,\frac{\varepsilon}{8}\right)-\log N_Z\left(\frac{\varepsilon}{4}\right).
\end{align*}
Hence,
\begin{align}\label{mdim-leq}
 \limsup_{\varepsilon\to0} \frac{h^S(\mathbb S,\varepsilon,\mathbb P)}{-\log\varepsilon}
    &\leq\limsup_{\varepsilon\to0}\left[\frac{h\left(T_G,\frac{\varepsilon}{8}\right)}{-\log\varepsilon}-\frac{\log N_Z\left(\frac{\varepsilon}{4}\right)}{-\log\varepsilon}\right]\\ \nonumber
    &=\displaystyle\overline{\text{mdim}}_M \,\Big(Y^\mathbb Y\times X,T_G, D\times d\Big)-\overline{\text{dim}}_B(Z)\\ \nonumber
    &=\displaystyle\overline{\text{mdim}}_M \,\Big(X,\mathbb S, d, \mathbb{P}\Big).
\end{align}
By \eqref{mdim-geq} and \eqref{mdim-leq} we obtain the result.

\subsection{Proof of Theorem \ref{thm8}}
let $\nu_v$ be the natural volume  measure on $X$ and assume that 
$\displaystyle\underline{\text{mdim}}_{\nu_v} \,(x, d)\geq s $, for all $x\in X$. Fix $\eta>0$ and let 
$$
X_k=\left\{x\in X: \frac{\limsup_{n\to\infty}-\frac{1}{n}\log\nu(B_n^G(x,\varepsilon))}{-\log\varepsilon}> (s-\delta\slash2) \text{ for all }\varepsilon\in(0,\frac{1}{k})\right\}.
$$
By hypotheses, $X=\bigcup_{k\in\mathbb N}X_k$. For $\varepsilon\in(0,\frac{1}{5\cdot k}]$ and $x\in X_k $, there exists $n(x)\in\mathbb N$ so that 
for any $N\geq n(x)$ we have
$$
\nu_v(B_n^G(x,\varepsilon))\geq e^{-(s+\delta)N\cdot -\log\varepsilon}.
$$
Since $X$ is a compact Riemannian manifold it has bounded geometry (see \cite{E} for more details on manifolds of bounded geometry). It implies that each function 
$f_m:X_k\to\mathbb R$ given by $f_m(x):=\nu_v(B_m^G(x,\varepsilon))$ is continuous and so
$$
N_0:=\sup\{n(x):x\in X_k\}<\infty.
$$
By Vitali Covering Lemma, for any $N\geq N_0 $ it is possible to choose from the cover $\mathcal B_N:=\{\overline{B_N^G(x,\varepsilon)}:x\in X_K\}$
of $\overline{X_k}$ a subset $F_N\subset X_k$ and a family  $\mathcal D_N:=\{\overline{B_N^G(x,\varepsilon)}:x\in F_N\}$ of disjoint balls for which we have
$$
X_k\subset\overline{X_k}\subset \bigcup_{x\in F_N}\overline{B_N^G(x,5\varepsilon)}\subset  \bigcup_{x\in F_N}B_N^G(x,6\varepsilon)
$$
and 
$$
\nu_v(B_N^G(x,\varepsilon))\geq e^{-(s+\delta)N\cdot -\log\varepsilon} \text{  for all }x\in F_N.
$$
So, as the family $\mathcal D_N$ is given by disjoint balls,
$$
\sharp(F_N)\cdot e^{-(s+\delta)N\cdot-\log\varepsilon}=\sum_{x\in F_N}e^{-(s+\delta)N\cdot-\log\varepsilon}\leq \sum_{x\in F_N}\nu_v(B_N^G(x,\varepsilon))\leq1.
$$
As 
$$
h_{GLW}(X_k,\mathbb S,6\varepsilon)\leq \limsup_{n\to\infty}\frac{1}{N}\log \sharp(F_N)
$$
we have
$$
\sup_{k\in\mathbb N}\displaystyle\overline{\text{mdim}}_M^{GLW}\,(X_k,\mathbb S, d)=\limsup_{\varepsilon\to0}\frac{h_{GLW}(X_k,\mathbb S,6\varepsilon)}{-\log\varepsilon}\leq s-\delta 
$$
since $X_{k}\subset X_{k+1}$ for all $k\in \mathbb N$ and $X=\bigcup_{k\in\mathbb N} X_k$ we have
$$\displaystyle\overline{\text{mdim}}_M^{GLW}\,(X,\mathbb S, d)\leq s-\delta.$$
As $\delta\geq0$ may be considered arbitrary small we have  
$$\displaystyle\overline{\text{mdim}}_M^{GLW}\,(X,\mathbb S, d)\leq s$$
and it finishes the proof.
\subsection{Proof of Theorem \ref{thm7}}
The following lemma is an important tool in the proof.

\begin{lemma}\label{lemma:lemma-homogeneity}
Let $\nu\in\mathcal M(X)$ be a G-homogeneous probability measure. Then
$$
\displaystyle\overline{\text{mdim}}_\nu\,(x, d) =\displaystyle\overline{\text{mdim}}_\nu\,(y, d), \text{ for all } x,y\in X.
$$
\end{lemma}
\begin{proof}
For $\varepsilon>0$, by $G$-homogeneity, there exists $\delta(\varepsilon)>0$ and $c>0$ so that 
$$
\nu(B_n^G(x,\delta(\varepsilon)))\leq c\cdot\nu(B_n^G(y,\varepsilon)) ,
$$
and it implies 
$$
h_{\nu}^G(x,\delta(\varepsilon))=\limsup_{n\to\infty}-\frac{1}{n}\log\nu(B_n^G(x,\delta(\varepsilon)))\leq\limsup_{n\to\infty}-\frac{1}{n}\log\nu(B_n^G(y,\varepsilon))=h_{\nu}^G(y,\varepsilon) ,
$$
and so
$$
\limsup_{\varepsilon\to0}\frac{ h_{\nu}^G(x,\delta(\varepsilon))}{-\log \delta(\varepsilon)} \leq\limsup_{\varepsilon\to0}\frac{ h_{\nu}^G(y,\varepsilon)}{-\log \varepsilon}, \text{ for all } x,y\in X,
$$
which gives
$\displaystyle\overline{\text{mdim}}_\nu\,(x, d)\leq \displaystyle\overline{\text{mdim}}_\nu\,(y, d)$.
By switching the roles of $x$ and $y$ in the previous computations one obtains the converse inequality and finishes the proof.
\end{proof}
As a consequence of Lemma \ref{lemma:lemma-homogeneity} we obtain that makes sense to define the measure metric mean dimension of a semigroup action with to respect of a $G$-homogeneous measure as the following:
\begin{align*}
    \displaystyle\overline{\text{mdim}}_\nu\,(\mathbb S, d)=\limsup_{\varepsilon\to0}\frac{ h_{\nu}^G(x,\varepsilon)}{-\log \varepsilon}, \text{ for any }x\in X
\end{align*}
since the limsup considered is constant in $X$.

\begin{proposition}\label{prop:mdim_GLW}
Let $G$ be a compactly generated semigroup  and $\nu$ be a strongly $G$-homogeneous probability measure on a  compact metric space $(X,d)$. Then 
$$
\displaystyle\overline{\text{mdim}}_\nu\,(X,\mathbb S, d)=\displaystyle\overline{\text{mdim}}_M^{GLW}\,(X,\mathbb S, d).
$$
\end{proposition}
\begin{proof}
Fix $\varepsilon>0$ and take $E$ a maximal $(n,\varepsilon)$-separated set in $X$. Then, by the maximality property of $E$,
$B_n^G(x,\varepsilon\slash2)\cap B_n^G(y,\varepsilon\slash2)$ for any $x,y\in E$. In particular, for a fixed $x\in E$
$$
\nu(X)\geq \sum_{y\in E}\nu\left(B_n^G(y,\varepsilon\slash2)\right)\geq s(n,\varepsilon)\cdot\nu(B_n^G(x,\varepsilon\slash2)).
$$
By the  $G$-homogeneity there exist $0<\delta(\varepsilon)<\varepsilon$ and $c>0$ so that 
$\nu(B_n^G(y,\delta(\varepsilon)))\leq c\cdot \nu(B_n^G(x,\varepsilon\slash2))$, for all $x,y\in X$.
It follows that 
$$
\limsup_{n\to\infty}\frac{1}{n}\log s(n,\varepsilon)\leq \limsup_{n\to\infty}-\frac{1}{n}\log\nu(B_n^G(y,\delta(\varepsilon))).
$$
Now, by the strongly $G$-homogeneity
\begin{align*}\label{eq:desi}
 \nonumber \displaystyle\overline{\text{mdim}}_M^{GLW}\,(X,\mathbb S, d)
  &= \limsup_{\varepsilon\to0} \frac{h_{GLW}(\mathbb S,\varepsilon)}{-\log\varepsilon}\\ \nonumber
  &\leq \limsup_{\varepsilon\to0}\frac{h^G_\nu(\delta(\varepsilon))}{-\log\delta(\varepsilon)}\frac{\log\delta(\varepsilon)}{\log\varepsilon}\\
   &=\displaystyle\overline{\text{mdim}}_\nu\,(X,\mathbb S, d).
\end{align*}
and then
$
\displaystyle\overline{\text{mdim}}_M^{GLW}\,(X,\mathbb S, d)\leq \displaystyle\overline{\text{mdim}}_\nu\,(\mathbb S, d)
$.

For the opposite inequality, fix $\delta>0$ and notice that that if $F$ is a $(n,\varepsilon)$-spanning set of minimal cardinality $b(n,\varepsilon)$, then $X\subset \bigcup_{x\in F}B_n^G(x,2\delta)$. Given $\varepsilon>0$ there exist $0<\delta(\varepsilon)<\varepsilon$ and $c>0$
for which
$$
\nu\left(B_n^G(x,2\delta(\varepsilon))\right)\leq c\cdot \nu\left(B_n^G(y,\varepsilon)\right) \text{ for all }x,y\in X \text{ and }n\in\mathbb N.
$$
It guarantees that
$$
c\cdot b(n,\delta(\varepsilon))\cdot\nu\left(B_n^G(y,\varepsilon)\right)\geq \nu(X)>0
$$
and so, by  the strong $G$-homogeneity, we have
\begin{align*}
 \displaystyle\overline{\text{mdim}}_M^{GLW}\,(X,\mathbb S, d)
 &=\limsup_{\varepsilon\to0} \frac{h_{GLW}(\mathbb S,\delta(\varepsilon))}{-\log\delta(\varepsilon)}\\
 &\geq  \limsup_{\varepsilon\to0}\frac{h^{GLW}_\nu(\delta(\varepsilon))}{-\log\varepsilon}\frac{\log\varepsilon}{\log\delta(\varepsilon)}\\
 &=\displaystyle\overline{\text{mdim}}_\nu\,(X,\mathbb S, d),
\end{align*}
and it ends the proof.
\end{proof}

Let us proceed to the proof of Theorem \ref{thm7}. For the first we notice that it is a consequence of Proposition \ref{prop:mdim_GLW}.
For part (b) let $\nu$ be a Borel measure on $X$ so that 
$\displaystyle\underline{\text{mdim}}_\nu \,(x, d)\geq s $, for all $x\in X$. Fix $\eta>0$ and let 
$$
X_k=\left\{x\in X: \frac{\limsup_{n\to\infty}-\frac{1}{n}\log\nu(B_n^G(x,\varepsilon))}{-\log\varepsilon}> (s-\delta\slash2) \text{ for all }\varepsilon\in(0,\frac{1}{k})\right\}.
$$
By hypotheses, $X=\bigcup_{k\in\mathbb N}X_k$. It follows that $0<\nu(X)\leq\sum_{k}\nu(X_k)$, which guarantees the existence of some $k_0\in\mathbb N$ for which we have
$\nu(X_{k_0})>0$. Again, we can wright $X_{k_0}=\bigcup_{N\in\mathbb N}X_{k_0,N}$ where
$$
X_{k_0,N}=\left\{x\in X_{k_0}: \frac{-\log\nu(B_n^G(x,\varepsilon))}{-n\log\varepsilon}> (s-\delta\slash2) \text{ for all }n\geq N\right\}.
$$
In such case, there exists $N_0\in\mathbb N$ for which $\nu(X_{k_0,N_0})>0$. In particular,
$$
\nu(B_n^G(x,\varepsilon))\leq e^{-n(s-\delta)\cdot( -\log\varepsilon)}, \text{ for all }x\in X_{k_0,N_0}, \varepsilon\in(0,\frac{1}{k}) \text{ and }n\geq N_0. 
$$
Now, for each integer $N\geq N_0$ consider the open cover of $X_{k_0,N_0}$ given by $\mathcal B_N=\{B_N^G(x,\varepsilon): x\in X_{k_0,N_0}\}$. 
In such case we have that for a subcover $\mathcal C$ of $\mathcal B_N$
$$
\inf_{\mathcal C}\cdot \sharp(\mathcal C)e^{-N(s-\delta)\cdot( -\log\varepsilon)}=\inf_{\mathcal C}\left\{\sum_{B_N^G(x,\varepsilon)\in\mathcal C}e^{-N(s-\delta)\cdot( -\log\varepsilon)}\right\}\geq \nu(X_{k_0,N_0}).
$$
As $\text{cov}(X,N,\varepsilon)\geq \text{cov}(X_{k_0,N_0},N,\varepsilon)$, for all $N$ and $\varepsilon>0$,
we have
$$
\text{cov}(X,N,\varepsilon)e^{-N(s-\delta)\cdot( -\log\varepsilon)}\geq\nu(X_{k_0,N_0}),
$$
and it implies that
\begin{align*}
    \limsup_{N\to\infty}\frac{1}{N}\log \text{cov}(X,N,\varepsilon)e^{-N(s-\delta)\cdot( -\log\varepsilon)}\geq 0 
\end{align*}
and so,
$$
h_{GLW}(X,\mathbb S,\varepsilon)\geq (s-\delta)\cdot( -\log\varepsilon).
$$
Hence
$$
\displaystyle\overline{\text{mdim}}_M^{GLW}\,(X,\mathbb S, d)=\limsup_{\varepsilon\to 0}\frac{ h_{GLW}(X,\mathbb S,\varepsilon)}{-\log\varepsilon}\geq s-\delta.
$$
As the inequality was obtained for an arbitrary $\delta$ we conclude that 
$$\displaystyle\overline{\text{mdim}}_M^{GLW}\,(X,\mathbb S, d)=\limsup_{\varepsilon\to 0}\frac{ h_{GLW}(X,\mathbb S,\varepsilon)}{-\log\varepsilon}\geq s,
$$
as part (b) states.

\end{document}